\documentclass[10pt,a4paper]{article}

\usepackage{amsmath,amssymb,multirow,color,subfig,amsfonts,amsbsy,mathtools,tikz,float}
\usepackage{graphicx}
\usepackage{hyperref}
\hypersetup{
	colorlinks = true,
	linkcolor = {blue},
	urlcolor = {red},
	citecolor = {blue}
}
\usepackage[nameinlink,capitalise,sort]{cleveref}
\usepackage{todonotes}

\usepackage{enumitem}

\newcommand{\R}{\mathbb{R}}
\newcommand{\RR}{\mathbb{R}}
\newcommand{\N}{{\mathbb N}}
\newcommand{\NN}{{\mathbb N}}

\newcommand{\PP}{\mathcal{P}}

\newcommand{\eqdef}{\vcentcolon=}
\newcommand{\dd}[1]{\mathinner{\mathrm{d}{#1}}}

\newcommand{\norm}[1]{\left\| #1 \right\|}
\newcommand{\abs}[1]{\left| #1 \right|}

\newcommand{\plan}[2]{\mathcal{T}\left(#1, #2\right)}
\newcommand{\optplan}[2]{\mathcal{T}^{opt}\left(#1, #2\right)}
\newcommand{\pushforward}{{}_{\#}}

\newcommand{\rvf}{\mathbf{RVF}}
\newcommand{\mvf}{\mathbf{MVF}}
\newcommand{\rfc}{\mathbf{RFC}}

\newtheorem{theorem}{Theorem}
\newtheorem{prop}{Proposition}
\newtheorem{definition}{Definition}

\newtheorem{corollary}{Corollary}
\newtheorem{remark}{Remark}

\def\MyItem[#1]#2{\item[{\rm #1}]#2}
\newenvironment{proof}{\noindent{\textsc{Proof.}}} {$\hfill\Box$\vspace{0.1 cm}\\}

\DeclareMathOperator{\lip}{Lip}
\DeclareMathOperator{\spt}{spt}
\DeclareMathOperator{\diam}{diam}

\begin{document}
\title{Properties of Measure Controls\\ and Their Trajectories}
\author{Mauro Garavello$^1$, Xiaoqian Gong$^2$, and Benedetto Piccoli$^3$}

\footnotetext[1]{Department of Mathematics and its Applications,
  University of Milano-Bicocca,
  Via  R. Cozzi 55, I--20125 Milano, Italy.\hfill\\
  Email address: \texttt{mauro.garavello@unimib.it}
}

\footnotetext[2]{Department of Mathematics,
  Amherst College,
  Amherst, MA, USA, 01002. \hfill\\
  Email address: \texttt{xgong@amherst.edu}
}

\footnotetext[3]{Department of Mathematical Sciences \&  Center for Computational and Integrative Biology, Rutgers University Camden, Camden, NJ, USA, 08102. \hfill\\
    Email address: \texttt{piccoli@camden.rutgers.edu}}

\maketitle

\begin{abstract}
    This paper deals with the concepts of \textsl{measure controls} 
    and of \textsl{measure vector fields},
    within the mathematical framework of \textsl{measure differential equations} (MDEs),
    recently proposed in~\cite{piccoli_measure_2019}. 
    Measure controls can be seen as a generalization of relaxed control.
    Moreover, they are particularly suitable for studying dynamics with uncertainty.

    The main results of this paper include establishing the existence and 
    well-posedness of control
    systems with measure controls and proving the equivalence between measure controls
    and measure vector fields. The stability and closure properties of the 
    trajectory set are also studied.
    \medskip

   
   \medskip

   \noindent\textbf{MSC~2020:} 34H05, 35Q93, 34A60, 35R06.

\end{abstract}

\section{Introduction}
\label{sec:introduction}

Traditionally, the study of control systems has focused on deterministic settings, where state variables are governed 
by dynamics, influenced by a fixed set of inputs or controls \cite{Agrachev_book,bressan-piccoli-book}. 
However, many real-world applications, particularly in economics, 
biology, and robotics, involve control mechanisms subject 
to uncertainties and fluctuations, best described by probabilistic models. 
This consideration has led to the generalization of control systems 
through the inclusion of generalized control functions, as for example relaxed controls
\cite{artstein_1978},
which extend the notion of classical control inputs by using probability measures.

Inspired by the recent work in~\cite{piccoli_manzo},
this paper explores  
the concepts of \textsl{measure controls} and of \textsl{measure vector fields},
within the mathematical framework of \textsl{measure differential equations} (MDEs),
proposed in~\cite{piccoli_measure_2019}. 
A measure control, which can be seen as a direct generalization of relaxed controls
\cite{artstein_1978}, is
a map assigning to every probability measure $\mu$ a probability measure on the product of the state and the control space. 
Using disintegration \cite{ambrosio-gigli-savare-book}, 
a measure control is equivalent to assigning a probability measure on the control set $U$ for every point $x$ in the state space. The mass located at $x$ then moves in the directions of the probability measure on $U$. The approach uses the framework of optimal transport, and in particular Wasserstein distances \cite{Santambrogio_book,villani2008optimal}, and is effective to deal with problems with uncertainty. 
The paper~\cite{piccoli_manzo}
introduced the concept of measure control and provided some basic properties of measure controls. The main aim of the present paper is to prove various results concerning measure controls, in the same spirit as properties of controlled systems of ordinary differential equations \cite{bressan-piccoli-book}. In this sense, the present approach can be seen as a natural generalization of deterministic control system to a probability setting.

Measure controls were defined using the approach of MDEs, see \cite{piccoli_measure_2019}. This type of equations have been analyzed by many recent papers: \cite{camilli2020superposition} proved 
a superposition principle for MDEs; 
\cite{piccoli2018measure} generalized the concept to measure differential inclusions;
\cite{cavagnari2023dissipative} introduced a variational perspective for dissipative multifunctions for measures;
\cite{piccoli2019measure} considered MDEs with sources and applications to crowd dynamics were presented in \cite{PR18}; 
\cite{dull2023structured} explored applications to population biology, 
\cite{piccoli2019modeling} focused on disturbance rejection,
\cite{borghi2025dynamics} applications to multi-agent system, 
\cite{piccoli2023control} provided a general perspective, and \cite{gong2022measure} studied a coupled MDE with an ordinary differential equation (ODE) to model viral infections. A related approach is the one developed in \cite{bonnet2021differential} using multivalued maps between measures and vector fields. Other notable approaches to transport-type equations for measures and intersections with control theory can be found in
\cite{ambrosio-gigli-savare-book,cavagnari2018generalized,cavagnari2015hamilton,chen2021optimal,
d2024lyapunov,d2026differential,jimenez2020optimal}.

The main contributions of this paper range from establishing the existence and 
well-posedness of control
systems with measure controls, to proving the equivalence between measure controls
and measure vector fields. We also investigate the stability and closure properties of the 
trajectory set.
Specifically, for a fixed measure control, we prove that the associated differential equation
admits a unique measure-valued solution, constructed via the 
\textsl{Lattice Approximate Solution} method; see~\cite{piccoli_measure_2019}.
We prove that this solution is stable, exhibiting continuous dependence on the control
within the topology induced by the Wasserstein metric.
Furthermore, by applying classical compactness results for continuous functions, we establish 
the closure of the trajectory set, namely that
every limit point of a sequence of trajectories is itself a solution 
to an equation with a measure-valued vector field.
Finally, using techniques typical of differential inclusions \cite{aubin-cellina-book}, 
we show a one-to-one correspondence between measure controls 
and measure vector fields, that yield identical solutions.

The paper is organized as follows.
In \cref{sec:control_system} we recall the basic definitions of control systems and
differential inclusions, and we introduce the concepts of measure vector field and
measure controls.
In \cref{sec:equivalence_MVF_measurecontrols} we prove that the concept of measure vector
field is equivalent to that of measure control.
\cref{sec:Peano} and \cref{sec:well-posedness} deal respectively with an existence theorem
and a well-posedness result for the control system with measure controls.
The proof of stability of the trajectories is contained in \cref{sec:stability},
while in \cref{sec:closure} the property of closure for the trajectories is shown.
Finally, \cref{sec:Wasserstein} contains the basic notions of the Wasserstein metric on a 
Polish space.

\section{Control Systems}
\label{sec:control_system}
In this section we recall the basic notions for classical control systems on the Euclidean
space $\RR^n$, and we
introduce a generalization to the case of measure valued controls. For more
details see~\cite{Agrachev_book, bressan-piccoli-book, warga_book}.

Let $(\mathcal U, d_{\mathcal{U}})$ be a Polish space, that is, a metric space that is both complete and separable. For each point \(x \in \mathbb{R}^n\), let \(T_x \mathbb{R}^n\) denote the tangent space of \(\mathbb{R}^n\) at \(x\). The tangent bundle \(T\mathbb{R}^n\) is the collection of all tangent spaces \(T_x \mathbb{R}^n\) as \(x\) ranges over \(\mathbb{R}^n\). A classical control system
on $\RR^n$ is given by 
\begin{equation}
    \dot{x} = f(x, u), 
    \label{eqn: control}
\end{equation}
where $x \in \mathbb{R}^n$ denotes the system state, 
$u \in U$ is the control, $U \subset \mathcal{U}$ is a  
subset referred to as the control set, and $f \colon \mathbb{R}^n \times U \to T\mathbb{R}^n$ 
is the vector field.
We use the notation $f(x, u) = \left(x, \hat{f}(x, u)\right)$ where 
$\hat{f}$ is the section of the vector field $f$.
On the cartesian product $\mathbb{R}^n \times U$, we consider the distance $d_{\mathbb{R}^n \times U}$ defined by
\begin{equation*}
    d_{\mathbb{R}^n \times U}((x_1, u_1), (x_2, u_2)) = \|x_1 - x_2\|_{\RR^n} + d_{\mathcal{U}}(u_1, u_2), \, \forall (x_i, u_i) \in \R^n \times U, i=1, 2.  
\end{equation*}
Observe that for each fixed $u \in U$, there is a natural vector field $f_u$ associated to~\eqref{eqn: control}.
More precisely, $f_u$ is defined by
\begin{equation*}
    \begin{array}{rccc}
        f_u: & \R^n & \longrightarrow & T \R^n
        \\
        & x & \longmapsto & f(x, u).
    \end{array}
\end{equation*}
In the following, we identify the tangent bundle $T\mathbb{R}^n$ with $\mathbb{R}^n \times \mathbb{R}^n$, 
and, consequently, we denote by $d_{T\RR^n}$ and by $\norm{\cdot}_{T\RR^n}$
respectively the Euclidean distance and norm on $\mathbb{R}^n \times \mathbb{R}^n$. 

To demonstrate the key concepts without addressing regularity issues, we introduce the following
assumptions.
\begin{enumerate}[label=\bf{(H-\arabic*)}, ref=\textup{\textbf{(H-\arabic*)}},
  align=left]
    \item \label{ass:compact_controls}
    The control set \(U \subset (\mathcal{U}, d_{\mathcal{U}})\) is compact.
    \item \label{ass:control_system}
    The vector field \(f \colon \mathbb{R}^n \times U \to T\mathbb{R}^n\) is Lipschitz continuous 
    in both variables, i.e. 
    there exists \(L_f >0\) such that 
    \begin{equation}
        \label{eq:Lip-control-system}
        d_{T\RR^n} \left(f(x, u_1), f(y, u_2)\right) \leq L_f 
        \left(\norm{x-y}_{\RR^n} + d_{\mathcal{U}}(u_1, u_2)\right)  
    \end{equation}
    for every $x, y \in \R^n$ and $u_1, u_2 \in U$.
    \item \label{ass:convexity_control_system}
    For every $x \in \RR^n$ the set $\left\{\hat f(x, u) \colon u \in U\right\}$
    is convex.
\end{enumerate}

\subsection{Differential Inclusions}
The control system~\eqref{eqn: control} naturally induces a set-valued dynamics described by the following differential inclusion:
\begin{equation}
    \dot{x} \in F(x),\qquad \textrm{where} \qquad F(x) \eqdef \left\{\hat f(x, u), u\in U\right\}. 
    \label{eq:diff_inclusion}
\end{equation}
Note that the solution sets for \eqref{eqn: control} and \eqref{eq:diff_inclusion} are equivalent. 
We refer the readers to \cite[Theorem~3.1.1]{bressan-piccoli-book} for more details. 

To quantify the continuity properties of the multivalued map \(x \mapsto F(x)\), 
we recall the definitions of lower semicontinuity (see~\cite[Chapter~1]{aubin-cellina-book})
and of 
Hausdorff continuity (see~\cite[Appendix~A.7]{bressan-piccoli-book}) for a multifunction.
\begin{definition}
    \label{def:lower_sc}
    Consider a multivalued function $F \colon \RR^n \hookrightarrow \RR^n$.
    \begin{enumerate}
        \item Given $x_o \in \RR^n$, the map $F$ is said lower semicontinuous at $x_o$
            if, for every $v_o \in F(x_o)$ and every neighborhood $\mathcal N_2$ of $v_o$,
            there exists a neighborhood $\mathcal{N}_1$ of $x_o$ such that
            \begin{equation*}
                F(x) \cap \mathcal{N}_2 \ne \emptyset
            \end{equation*}
            for every $x \in \mathcal{N}_1$.

        \item The map $F$ is said lower semicontinuous if it is lower semicontinuous
            at every $x_o \in \RR^n$.
    \end{enumerate}
\end{definition}
We now introduce the definition of Hausdorff distance and Hausdorff continuity. 
\begin{definition}
    \label{def:Hausdorff_distance}
    Given two closed sets $C_1,C_2\subseteq \R^n$ we define
    the Hausdorff distance $d_H$ of $C_1$ from $C_2$ by
    \begin{equation}
        \label{eq:Hasudorff_distance}
        d_H(C_1,C_2) \eqdef \max \left\{\sup_{x \in C_1} d(x, C_2), \sup_{y \in C_2} d(y, C_1)\right\},    
    \end{equation}
   where, for a closed set $C \subseteq \mathbb{R}^n$ and a point $x \in \mathbb{R}^n$, the distance from $x$ to $C$ is given by
\[
d(x,C) \coloneqq \inf_{y \in C} \|x-y\|_{\mathbb{R}^n}.
\]
\end{definition}
\begin{definition}
    \label{def:Hausdorff_continuity}
    A multifunction $F \colon \RR^n \hookrightarrow \RR^n$ with compact values is said Hausdorff
    continuous if, for every $x_o \in \RR^n$,
    \begin{equation*}
        \lim_{x \to x_o} d_H(F(x), F(x_o)) = 0.
    \end{equation*}
\end{definition}

\begin{prop}
    \label{prop:multi_Hausdorff}
    Assume~\ref{ass:compact_controls} and~\ref{ass:control_system}.
    Then, the map $x \mapsto F(x)$, defined in~\eqref{eq:diff_inclusion}, 
    is Lipschitz continuous for the Hausdorff metric $d_H$, 
    that is
    \begin{equation*}
        d_H(F(x),F(y))\leq L_f \norm{x-y}_{\RR^n}
    \end{equation*}
    for every $x, y \in \R^n$, where $L_f$ is the Lipschitz constant of $f$.
\end{prop}
\begin{proof}
    First note that, by~\ref{ass:compact_controls} and~\ref{ass:control_system}, for every $x \in \R^n$, the
    set $F(x) \subset \R^n$ is closed. Denote with $L_f$ the Lipschitz constant of the
    vector field $f$; see~\ref{ass:control_system}.

    Fix arbitrarily $x, y \in \R^n$ and $z_1 \in F(x)$. 
    There exists $u_1 \in U$ such that $\hat f(x, u_1) = z_1$.
    Hence, by~\ref{ass:control_system},
    \begin{align*}
        \norm{z_1 - \hat f(y, u_1)}_{\RR^n} = \norm{\hat f(x, u_1) - \hat f(y, u_1)}_{\RR^n} 
        \le L_f \norm{x-y}_{\RR^n}                
    \end{align*}
    and so
    \begin{align*}
        d(z_1, F(y))
        & = \inf_{u \in U} \norm{z_1 - \hat f(y, u)}_{\RR^n} 
        \le \norm{\hat f(x, u_1) - \hat f(y, u_1)}_{\RR^n} 
        \le L_f \norm{x- y}_{\RR^n}.
    \end{align*}
    The arbitrariness of $z_1 \in F(x)$ implies that
    \begin{equation*}
        \sup_{z \in F(x)} d(z, F(y)) \le L_f \| x-y \|_{\RR^n}.
    \end{equation*}
    Similarly, we get that
    \begin{equation*}
        \sup_{z \in F(y )}d(z, F(x))
        \le L_f \| x- y\|_{\RR^n}.
    \end{equation*}
    Therefore, we obtain that
    \begin{equation*}
        d_H(F(x), F(y)) = \max \left\{\sup_{z \in F(x)} d(z, F(y)), \sup_{z \in F(y)} d(z, F(x))\right\}
        \le L_f \| x-y\|_{\RR^n}.
    \end{equation*}
    This concludes the proof.
\end{proof}

\begin{prop}
    \label{prop:Hausdorff_lsc}
    Consider a Hausdorff continuous multifunction $F:\R^n\hookrightarrow \R^n$.
    Then, $F$ is lower semicontinuous.
\end{prop}
\begin{proof}
    Fix $x_o \in \RR^n$, $v_o \in F(x_o)$, and a neighborhood $\mathcal{N}_2$ of $v_o$.
    There exists $\epsilon > 0$ such that $B(v_o, \epsilon) \subseteq \mathcal{N}_2$.
    The Hausdorff continuity of $F$ implies the existence of $\delta > 0$ such that
    $d_H(F(x), F(x_o)) < \epsilon$ for every $x \in B(x_o, \delta)$.
    Therefore, $\sup_{z \in F(x_o)} d(z, F(x)) < \epsilon$ for every $x \in B(x_o, \delta)$,
    which implies that $d(v_o, F(x)) < \epsilon$ for every $x \in B(x_o, \delta)$.
    Thus, for every $x \in B(x_o, \delta)$, there exists $v \in F(x)$ 
    such that $d(v_o, v) < \epsilon$.
    This implies that, for every $x \in B(x_o, \delta)$, there exists $v \in F(x)$ 
    such that $v \in \mathcal{N}_2$. This concludes the proof.
\end{proof}

\begin{corollary}
    \label{cor:F_lower_sc}
    Assume~\ref{ass:compact_controls} and~\ref{ass:control_system}.
    Then, the map $x \mapsto F(x)$, defined in~\eqref{eq:diff_inclusion}, 
    is lower semicontinuous.
\end{corollary}

\begin{proof}
    It is a direct consequence of \cref{prop:multi_Hausdorff} and    
    \cref{prop:Hausdorff_lsc}.
\end{proof}

\subsection{Relaxed Controls}

Here we recall the concepts of relaxed controls and of relaxed feedback controls;
see~\cite{artstein_1978, warga_book} for more details.
\begin{definition}
    A \textbf{relaxed control} \(\tilde{u}\) is a probability measure on the control set $U$, 
    that is, \(\tilde{u} \in \mathcal{P}(U)\). 

    A \textbf{relaxed feedback control}, briefly $\rfc$, is a map \(\tilde{u} \colon \R^n \to \mathcal{P}(U)\). 
\end{definition}
Note that, for every fixed relaxed control \(\tilde{u} \in \mathcal{P}(U)\), 
there is a related vector field \(f_{\tilde{u}}\), namely 
\begin{equation*}
    \begin{array}{rccc}
        f_{\tilde u}\colon & \R^n & \longrightarrow & T \R^n
        \\
        & x & \longmapsto & \displaystyle \int_U f(x, u) \dd{} \tilde u(u).
    \end{array}
\end{equation*}
Moreover, given a relaxed feedback control \(\tilde{u}(\cdot)\), 
we can associate a vector field $f_{\tilde u(\cdot)}$, defined by
\begin{equation*}
    \begin{array}{rccc}
        f_{\tilde u(\cdot)}\colon & \R^n & \longrightarrow & T \R^n
        \\
        & x & \longmapsto & \displaystyle \int_U f(x, u) \dd{} (\tilde u(x))(u).
    \end{array}
\end{equation*}
The following result shows that the Lipschitz continuity of a relaxed feedback control 
with respect to the Wasserstein distance ensures the Lipschitz continuity of its corresponding vector field.
\begin{prop}
   Assume~\ref{ass:compact_controls} and~\ref{ass:control_system}. 
   Fix a relaxed feedback control \(\tilde{u}\) for~\eqref{eqn: control}. 
   Assume that \(\tilde{u}\) is Lipschitz continuous with respect to Wasserstein distance,
   i.e. there exists $L > 0$ such that, for every $x, y \in \R^n$,
   \begin{equation*}
        W_{\mathcal{U}}(\tilde u(x), \tilde u(y)) \le L \| x - y \|_{\RR^n}.
   \end{equation*}
   
   Then the vector field \(x \mapsto f_{\tilde{u}(\cdot)}(x)\) is locally Lipschitz continuous. 
\end{prop}
\begin{proof}
    Fix $x, y \in \R^n$. Using~\eqref{eq:K-R_duality} and the fact that $d_{T\RR^n}$ is
    the Euclidean distance in $\RR^n \times \RR^n$, we have
    \begin{align*}
        d_{T\RR^n} \left(f_{\tilde u(\cdot)}(x), f_{\tilde u(\cdot)}(y)\right)
        & = \norm{f_{\tilde u(\cdot)}(x) - f_{\tilde u(\cdot)}(y)}_{\RR^{2n}}
        \\
        & \le \int_U d_{T\RR^n}\left(f(x, u), f(y, u)\right) \dd{}(\tilde u(x))(u)
        \\
        & \quad + \norm{\int_U f(y, u) \dd{}(\tilde u(x) - \tilde u(y))(u)}_{\RR^{2n}}
        \\
        & {\le} L_f \norm{x-y}_{\RR^n} 
        + \sum_{i=1}^n \abs{\int_U y_i \dd{}(\tilde u(x) - \tilde u(y))(u)}
        \\
        & \quad + \sum_{i=1}^n \abs{\int_U \hat f_i(y, u) \dd{}(\tilde u(x) - \tilde u(y))(u)}
        \\
        & {\le} L_f \norm{x-y}_{\RR^n} 
        + \sum_{i=1}^{n} (\abs{y_i} + L_f) W_{\mathcal{U}}(\tilde u(x), \tilde u(y))
        \\
        & {\le} L_f \norm{x-y}_{\RR^n} + (\sqrt{2}\norm{y}_{\RR^n} + n L_f) L \norm{x-y}_{\RR^n},
    \end{align*}
    where we denoted by $L_f$ the Lipschitz constant of $f$.
    Therefore, the vector field $x \mapsto f_{\tilde{u}(x)}(x)$ is locally Lipschitz continuous.
\end{proof}

\subsection{Random and Measure Vector Fields}
To generalize the concept of relaxed feedback controls, we need to introduce 
the definitions of random vector field ($\rvf$) 
and measure vector field ($\mvf$). 

\begin{definition}
    A $\rvf$ associated with the control system~\eqref{eqn: control} is a map
    \begin{equation*}
        \begin{array}{rccc}
            v: & \R^n & \longrightarrow & \mathcal{P}(T \R^n)
            \\
            & x & \longmapsto & v(x)
        \end{array}
    \end{equation*}
    such that, for every $x \in \R^n$, 
    \[\spt(v(x)) \subset \{(x,v) \colon v \in F(x)\} \subset T_x \R^n.\]
\end{definition}

\begin{definition}
    A $\mvf$ is a map
    \begin{equation*}
        \begin{array}{rccc}
            V: & \mathcal{P}(\R^n) & \longrightarrow & \mathcal{P}(T \R^n)
            \\
            & \mu & \longmapsto & V[\mu]
        \end{array}
    \end{equation*}
    such that, for every $\mu \in \mathcal{P}(\R^n)$, \(\pi_1 \pushforward (V[\mu]) = \mu\).
\end{definition}

\begin{remark}
    Let \(V \) be a $\mvf$. 
    For every \(\mu \in \mathcal{M}(\R^n)\), \Cref{thm:disintegration} implies that
    there exists a family of probability measures 
    \(\nu_x = \nu_x[V, \mu]\), $x \in \R^n$, such that \(V[\mu] = \mu \otimes_x \nu_x\).
    Note that in general the measures $\nu_x$ depend both on $V$ and on $\mu$. 
\end{remark}

\begin{definition} 
    \label{def:mcf_associated}
    A $\mvf$ \(V \colon \mathcal{P}(\R^n) \to \mathcal{P}(T \R^n)\) is said     
    associated with the control system~\eqref{eqn: control} 
    if, for every $\mu \in \mathcal{P}(\R^n)$,
    \begin{equation*}
        \spt(V[\mu]) \subset \{(x,v) \in T\R^n \colon v \in F(x)\}.  
    \end{equation*}
\end{definition}

The following existence result for a measure differential equation holds; 
see~\cite[Theorem~3.1]{piccoli_measure_2019}.
\begin{theorem}
    \label{thm:existence-MDE}
    Consider $V: \mathcal{P}(\R^n) \to \mathcal{P}(T\R^n)$ a $\mvf$ satisfying the following properties.
    \begin{enumerate}
        \item There exists $C>0$ such that, for every $\mu \in \mathcal{P}(\R^n)$ with
        compact support,
      \begin{equation}
            \sup_{(x, v) \in \, \spt(V[\mu])} \norm{v}_{\RR^n} \le C \left(
                1 + \sup_{x \in \, \spt(\mu)} \norm{x}_{\RR^n}
            \right).
            \label{eqn: sublinear}
        \end{equation}

        \item For every $R>0$ the map
        $V$, restricted to the $\mathcal{P}^R_c(\RR^n)$ of probability measures with 
        support in $B(0, R)$, is continuous. In other words, for every $\bar \mu \in \mathcal{P}^R_c(\RR^n)$
        and $\varepsilon > 0$, there exists $\delta > 0$ such that
        \begin{equation*}
            W_{T\RR^n}\left(V[\mu], V[\bar \mu]\right) < \varepsilon
        \end{equation*}
        for every $\mu \in \mathcal{P}^R_c(\RR^n)$ with the property 
        $W\left(\mu, \bar \mu\right) < \delta$.
    \end{enumerate}
    Then, for every $\mu_o \in \mathcal{P}_c(\RR^n)$ and $T>0$, 
    there exists a solution $\mu:[0, T] \to \mathcal{P}_c(\RR^n)$ to $\dot \mu = V[\mu]$, $\mu(0) = \mu_o$,
    obtained as a limit of \textsl{Lattice Approximate Solution} (LAS).
    Moreover, if $\spt(\mu_o) \subseteq B(0, R)$ for some $R>0$, then, there exists \(C>0\) such that, for every $t, s \in [0, T]$,
    \begin{equation*}
        W_{\R^n}\left(\mu(s), \mu(t)\right) \le C \exp (CT) (R+1) \abs{t-s}.
    \end{equation*}
\end{theorem}

The following well-posedness result for measure differential equations hold;
see~\cite[Theorem~4.1]{piccoli_measure_2019}
\begin{theorem}
    \label{thm:wellposedness_MDE}
    Consider $V: \mathcal{P}(\R^n) \to \mathcal{P}(T\R^n)$ a $\mvf$ such that:
    \begin{enumerate}
        \item there exists $C>0$ such that, for every $\mu \in \mathcal{P}(\R^n)$ with
        compact support, \cref{eqn: sublinear} holds. 
        \item for every $R>0$ there exists $K > 0$ such that
        \begin{equation}
            \mathcal{W}_{T\RR^n}\left(V[\mu], V[\nu]\right) \le K W_{\RR^n}(\mu, \nu)
        \end{equation}
        for every $\mu, \nu \in \mathcal{M}(\R^n)$, $\mu(\R^n) = \nu(\R^n)$,
        $\spt(\mu) \subseteq B(0, R)$, $\spt(\nu) \subseteq B(0, R)$.
    \end{enumerate}
    Then there exists a Lipschitz semigroup of solutions of $\dot \mu = V[\mu]$.
\end{theorem}

\subsection{Measure Controls}
We introduce here the notion of measure control and of measure valued field associated
to a measure control.
Moreover, we define some regularity concepts for measure controls.

\begin{definition}
    \label{def:measure_control}
    A measure control is a map
    \begin{equation*}
        \tilde u: \mathcal{P}(\RR^n) \longrightarrow \mathcal{P}(\RR^n \times U)
    \end{equation*}
    such that, for every $\mu \in \mathcal{P}(\RR^n)$,
    $\pi_1 \pushforward (\tilde u[\mu]) = \mu$.
\end{definition}

\begin{remark}
    Let \(\tilde u \) be a measure control. 
    For every \(\mu \in \mathcal{P}(\R^n)\), \cref{thm:disintegration} implies that
    there exists a family of probability measures 
    \(\nu_x = \nu_x[\tilde u, \mu]\), $x \in \R^n$, such that \(\tilde u[\mu] = \mu \otimes_x \nu_x\).
    Note that in general the measures $\nu_x$ depend both on $\tilde u$ and on $\mu$. 
\end{remark}

If $\tilde{u}$ is a measure control, then we define the $\mvf$ $V^{\tilde{u}}$ 
associated to $\tilde u$ and to the control system~\eqref{eqn: control} by:
\begin{equation}
    \label{eq:mfv-given-tilde-u}
    \begin{array}{rccc}
        V^{\tilde u}:
        & \mathcal{P}(\R^n)
        & \longrightarrow
        & \mathcal{P}(T \R^n)
        \\
        & \mu
        & \longmapsto
        & f \pushforward (\tilde u[\mu]).
    \end{array}
\end{equation}

Finally, we introduce the notion of $\mathcal W$-continuity for a measure control and for
a sequence of measure controls.

\begin{definition}
    \label{def:regularity_measure_control}
    Let $\tilde u\colon \PP(\RR^n) \to \PP(\RR^n \times U)$ be a measure control.
    \begin{enumerate}
        \item The measure control $\tilde{u}$ is $\mathcal{W}$-continuous at 
            the point $\mu\in\mathcal P_c(\R^n)$ if, for every 
            sequence $\mu_k\in\mathcal{P}_c(\R^n)$ such that 
            $W_{\RR^n}(\mu_k,\mu)\to 0$ as $k\to\infty$, it holds
            \begin{equation*}
                \lim_{k\to+\infty} \mathcal{W}_{\RR^n \times U} (\tilde u[\mu_k], \tilde u[\mu]) = 0.
            \end{equation*}

        \item The measure control $\tilde{u}$ is $\mathcal{W}$-continuous if it is at 
            every point $\mu\in\mathcal P_c(\R^n)$.

        \item The measure control $\tilde{u}$ is said $\mathcal{W}$-Lipschitz continuous if
            there exists a constant $C > 0$ such that
            \begin{equation*}
                \mathcal{W}_{\RR^n \times U}\left(\tilde u[\mu], \tilde u[\nu]\right) \le 
                C\, W_{\RR^n}(\mu, \nu)
            \end{equation*}
            for every $\mu, \nu \in \mathcal P_c(\RR^n)$.
    \end{enumerate} 
\end{definition}

\begin{definition}
    \label{def:unif-Lip-measure-controls}
    Let $\tilde u_k\colon \PP(\RR^n) \to \PP(\RR^n \times U)$, $k \in \NN$, be a 
    sequence of measure controls.

    The sequence of measure controls $\tilde{u}_k$
    is said uniformly $\mathcal{W}$-Lipschitz continuous if  
    there exists a constant $C>0$ such that, for every $k\in\N$ and $\mu,\nu \in\PP_c(\R^n)$, it holds
    \begin{equation}
        \label{eq:W-Lip2}
        \mathcal{W}_{\R^n \times U}
        ({\tilde{u}_k}[\mu], {\tilde{u}_k}[\nu])
        \leq C\, W_{\RR^n}(\mu,\nu).
    \end{equation}
\end{definition}

\section{Equivalence of \texorpdfstring{$\mvf$}{MVF} and Measure Controls}
\label{sec:equivalence_MVF_measurecontrols}
In this section, we prove that, given a $\mvf$ $V\colon \PP(\RR^n) \to \PP(T\RR^n)$,
associated to the control system~\eqref{eqn: control},
it is possible to construct a measure control $\tilde u$ such that
$V = V^{\tilde u}$ on probability measures with compact support.
Since to every measure control $\tilde u$ it is possible to construct $V^{\tilde u}$
as in~\eqref{eq:mfv-given-tilde-u}, we deduce that measure controls are equivalent to
measure vector field.

\begin{theorem}
    Assume hypotheses~\ref{ass:compact_controls}, \ref{ass:control_system}, 
    and~\ref{ass:convexity_control_system} hold.
    Consider a $\mvf$ $V$ associated with the control system~\eqref{eqn: control}, 
    in the sense of \Cref{def:mcf_associated}, 
    such that
    $V$ maps $\PP_c(\R^n)$ to $\PP_c(T\R^n)$.

    Then, for every $R>0$,
    there exists a measure control
    $\tilde{u}$ such that $V[\mu]=V^{\tilde{u}}[\mu]$ for every $\mu\in\PP_c(\R^n)$ with $\spt(\mu)\subset B(0,R)$.
\end{theorem}

\begin{proof}
    Fix $\epsilon>0$ and a compact set $K\subset\R^n$.

    Since $U$ is compact by~\ref{ass:compact_controls}, then, for every $x \in K$, the set
    $F(x) \eqdef \left\{\hat f(x, u) \colon u \in U\right\}$
    is compact since~\ref{ass:control_system} and so 
    totally bounded (see~\cite[Theorem~45.1]{munkres2013topology}), which
    implies the existence of elements
    $v_{x, 1}, \ldots, v_{x, N_{x}} \in F(x)$ such that
    $v_{x, i} \ne v_{x, j}$ for every $i \ne j$ and
    \begin{equation*}
        F(x) \subseteq \bigcup_{i=1}^{N_{x}} B\left(v_{x, i}, \frac{\epsilon}{2}\right).
    \end{equation*}
    Define, for every $i=1, \ldots, N_{x}$, the ($x$-dependent) set
    \begin{equation*}
        Z_i^x \eqdef \left\{v \in F(x): \norm{v - v_{x, i}}_{\RR^n} \le 
        \norm{v - v_{x, j}}_{\RR^n} \textrm{ for all } j\ne i\right\}.
    \end{equation*}
    Note that the sets $Z_1^x, \ldots, Z_{N_{x}}^x$ define the Voronoi
    partition of $F(x)$ associated to $v_{x, 1}, \ldots, v_{x, N_x}$.
    We have that, for every $i = 1, \ldots, N_x$, $\diam(Z_i^x) \le \epsilon$.
    Indeed, fix $w \in Z_i^x$. Hence, there exists
    $j \in \left\{1, \ldots, N_{x}\right\}$ such that $\norm{w - v_{x, j}}_{\RR^n} < \frac{\epsilon}{2}$ and so
    $\norm{w - v_{x, i}}_{\RR^n} \le \norm{w - v_{x, j}}_{\RR^n} < \frac{\epsilon}{2}$. 
    Therefore, $Z_i^x \subseteq B\left(v_{x, i}, \frac{\epsilon}{2}\right)$ and so
    $\diam(Z_i^x) \le \epsilon$.

    By~\ref{ass:compact_controls}-\ref{ass:convexity_control_system}, the
    map $F$ has compact and convex values. Moreover, by \cref{cor:F_lower_sc}, $F$
    is lower semicontinuous. Hence, 
    the Michael Selection Theorem~\cite[Chapter~1.11]{aubin-cellina-book} implies that, 
    for every $x \in K$, there exist $N_x$ continuous selections
    $\phi_{x, i}$ of $F$ ($i=1, \ldots, N_x$), both defined on 
    an open ball centered at $x$, denoted by $B_x$, such that:
    \begin{enumerate}
        \item $\phi_{x, i}(x) = v_{x, i}$ for every $i \in \left\{1, \ldots, N_x\right\}$;
        \item $\phi_{x, i}(y) \ne \phi_{x, j}(y)$ for every $j \ne i$ and $y \in B_x$;
        \item $\phi_{x, i}(y) \in F(y)$ for every $i \in \left\{1, \ldots, N_x\right\}$ 
        and $y \in B_x$;
        \item for every $y \in B_x$, the Voronoi partition
    $Z_{x, 1}^y, \ldots, Z_{x, N_x}^y$ of $F(y)$, associated to the points
    $\phi_{x, 1}(y), \ldots, \phi_{x, N_x}(y)$, satisfies
    \begin{equation}
        \label{eq:diam-2eps}
        \diam(Z_{x, i}^y)\leq 2\epsilon
    \end{equation}
    for every $i =1, \ldots, N_x$.
    \end{enumerate}

    The compactness of $K$ implies that there exist
    $\bar x_1, \ldots, \bar x_N$ so that
    $\displaystyle \bigcup_{j=1}^{N} B_{\bar x_j}$ covers $K$.
    Here the open balls $B_{\bar x_j}$ have the properties previously described.
    Using this choice, it is possible to construct, in unique way, an index map
    \begin{equation*}
        S \colon K \longrightarrow \left\{1, \ldots, N\right\}
    \end{equation*} 
    such that $x \in B_{\bar x_{S(x)}} \setminus \bigcup_{h=1}^{S(x)-1} B_{\bar x_h}$
    for every $x \in K$.

    Fix now $\mu \in \mathcal{P}(\RR^n)$ with $\spt(\mu) \subseteq K$.
    By \Cref{thm:disintegration}, there exist a subset $K_\mu$ of $K$ with $\mu(K_\mu) = 0$
    and, for every $x \in K \setminus K_\mu$, a unique probability measure $\nu_x[V, \mu]$ 
    living in $F(x) \subseteq T_x \R^n$ such that
    \begin{equation*}
        V[\mu] = \mu \otimes_x \nu_x[V, \mu].
    \end{equation*}
    For $x \in K \setminus K_\mu$, call $j = S(x)$ and define the atomic measure
    \begin{equation*}
        \alpha_x \eqdef \sum_{i=1}^{N_{\bar x_j}} m_{i, x} \, \delta_{\phi_{\bar x_j, i}(x)},
    \end{equation*}
    where, for every $i = 1, \ldots, N_{\bar x_j}$,
    \begin{equation*}
        m_{i, x} = \nu_{x}[V, \mu]\left(Z_{\bar x_j, i}^{x} 
        \setminus \bigcup_{h=1}^{i-1} Z_{\bar x_j, h}^x\right)    
    \end{equation*} 
    while $\delta_{v}$
    denotes the Dirac-delta measure centered at $v$.
    Thus, we have that 
    \begin{equation*}
        W_{T_x \RR^n}(\alpha_x, \nu_{x}[V,\mu])\leq \frac{\epsilon}{2} .    
    \end{equation*}
    Indeed, consider the transport map $\mathcal{T}: F(x) \to F(x)$,
    defined as $\mathcal{T}(v) = \phi_{\bar x_j, i}(x)$ 
    if $v \in Z_{\bar x_j, i}^x \setminus \bigcup_{j=1}^{i-1} Z_{\bar x_j, h}^x$. Note
    that $\mathcal{T} \pushforward (\nu_{x}[V, \mu]) = \alpha_x$. Hence,
    \begin{align*}
        W_{T_x \RR^n} (\alpha_x, \nu_{x}[V,\mu]) 
        & \le \int_{F(x)} \norm{v - \mathcal{T}(v)}_{\RR^n} 
        \dd \nu_{x}[V, \mu](v)
        \\
        & \le \sum_{i=1}^{N_{\bar x_j}} \int_{Z_{\bar x_j, i}^x 
        \setminus \bigcup_{h=1}^{i-1} Z_{\bar x_j, h}^x}
        \norm{v - \phi_{\bar x_j, i}(\bar x_j)}_{\RR^n} \dd \nu_{x}[V, \mu](v)
        \\
        & \le \frac{\epsilon}{2} \sum_{i=1}^{N_{\bar x_j}} m_{i, x}
        = \frac{\epsilon}{2} .
    \end{align*}

    For every $j \in \left\{1, \ldots, N\right\}$ and $i \in \left\{1, \ldots, N_{\bar x_j}\right\}$,
    define the multifunction $U_{j,i} \colon B_{\bar x_j} \hookrightarrow U$ as
    \[
        U_{j,i}(y) \eqdef \left\{u \in U \colon \hat f(y, u)=\phi_{\bar x_j, i}(y)\right\}.
    \]
    Notice that the multifunction $U_{j,i}$ is measurable because $\hat f$ and $\phi_{\bar x_j, i}$ 
    are continuous, thus we can define the lexicographic selection 
    $u_{j,i}\colon B_{\bar x_j} \to U$,
    which is also measurable (see~\cite[Theorem A.7.3]{bressan-piccoli-book}). 
    
    Then, for all $x \in K \setminus K_\mu$, given $j = S(x)$, define
    the atomic measure:
    \begin{equation*}
        \beta_x \eqdef \sum_{i=1}^{N_{\bar x_j}} m_{i, x} \, \delta_{u_{j,i}(x)},
    \end{equation*}
    which is a probability measure on $U$, i.e. $\beta_x \in \PP(U)$. Finally, 
    define a measure control $\tilde u_\epsilon$ by setting:
    \begin{equation*}
        \tilde{u}_\epsilon [\mu] \eqdef \mu \otimes_x \beta_x \, .
    \end{equation*}

    Passing to the limit as $\epsilon \to 0$ using the 
    Prokhorov Theorem~\cite[Theorem~5.1.3]{ambrosio-gigli-savare-book}, 
    we obtain a measure control $\tilde{u}$ defined on $K$ with the desired property.
    Given $R>0$, we apply the previous construction to the set $K = \overline{B(0,R)}$.
    This concludes the proof.
\end{proof}

\section{Peano-type Theorem for Measure Controls}
\label{sec:Peano}

In this section we consider the 
$\mvf$ $V^{\tilde u}$, defined in~\eqref{eq:mfv-given-tilde-u}, associated
to a given measure control $\tilde{u}$; see \cref{def:measure_control}. 
Under suitable assumptions on $\tilde u$, we prove existence of solution for the
Cauchy problem associated to $V^{\tilde{u}}$.

\begin{theorem}
    Assume~\ref{ass:compact_controls} and~\ref{ass:control_system}.
    Consider a $\mathcal{W}$-continuous
    measure control $\tilde{u}$, in the sense of \cref{def:regularity_measure_control}.
       
    Then,
    given $\mu_0\in\PP_c(\R^n)$ and $T>0$, there
    exists $\mu \colon [0, T] \to \mathcal{P}_c(\RR^n)$, solution to the Cauchy problem
    $\dot{\mu}=V^{\tilde{u}}[\mu]$, $\mu(0)=\mu_0$, 
    where the measure vector field $V^{\tilde{u}}$ is defined in~\eqref{eq:mfv-given-tilde-u}.
\end{theorem}

\begin{proof}
    We need to prove that the vector field $V^{\tilde{u}}$ satisfies
    the assumptions of \cref{thm:existence-MDE}.

    Fix $u_0 \in U$ and define
    \begin{equation*}
        C \eqdef \max \left\{L_f, L_f \diam(U) + \norm{\hat f(0, u_0)}_{\RR^n}\right\},
    \end{equation*}
    where $L_f$ is the Lipschitz constant for the vector field $f$; see~\ref{ass:control_system}.
    Note that $C \in \RR$ since~\ref{ass:compact_controls} and~\ref{ass:control_system}.
    Take $\mu \in \mathcal P_c(\RR^n)$ and $(x, v) \in \spt \left(V^{\tilde u}[\mu]\right)$.
    Then, there exists $u \in U$ such that $v = \hat f(x, u)$ and so, using~\ref{ass:control_system},
    \begin{align*}
        \norm{v}_{\RR^n}
        & = \norm{\hat f(x, u)}_{\RR^n} \le \norm{\hat f(x, u) - \hat f(0, u_0)}_{\RR^n} 
        + \norm{\hat f(0, u_0)}_{\RR^n}
        \\
        & \le L_f \norm{x}_{\RR^n} + L_f \diam (U) + \norm{\hat f(0, u_0)}_{\RR^n}
        \\
        & \le C \left(1 + \norm{x}_{\RR^n}\right)
        \\
        & \le C \left(1 + \sup_{x \in \spt\left(\mu\right)}\norm{x}_{\RR^n}\right),
    \end{align*}
    proving the first assumption of \cref{thm:existence-MDE}.

    Consider now the second assumption of \cref{thm:existence-MDE}.
    Fix $R > 0$ and consider the set
    \begin{equation*}
        \mathcal{P}^R_c (\RR^n) \eqdef \left\{\mu \in \PP_c(\RR^n) \colon \spt(\mu) \subseteq B(0, R)\right\}.
    \end{equation*}
    Take $\bar \mu \in \mathcal{P}^R_c (\RR^n)$ and
    a sequence $\mu_k \in \mathcal{P}^R_c (\RR^n)$ such that
    \begin{equation}
        \label{eq:mu_k_convege_to_bar_mu}
        \lim_{k \to + \infty} W_{\RR^n}(\mu_k, \bar \mu) = 0.
    \end{equation}
    We claim that
    \begin{equation}
        \label{eq:second_estimate_thm_existence}
        \lim_{k \to + \infty} W_{T\RR^n}\left(V^{\tilde u}[\mu_k], V^{\tilde u}[\bar \mu]\right) = 0.
    \end{equation}
    Fix $\varepsilon > 0$.
    Since~\eqref{eq:mu_k_convege_to_bar_mu} and since $\tilde u$ is $\mathcal{W}$-continuous at $\bar \mu$, 
    then there exists $\bar k \in \NN$ such that, for every $k \ge \bar k$,
    \begin{equation}
        W_{\RR^n}(\mu_k, \bar \mu) < \frac{\varepsilon}{2(L_f+1)}
    \end{equation}
    and there exists a plan
    $S_k \in \mathcal{T}\left(\tilde u[\mu_k], \tilde u[\bar \mu]\right)$ satisfying
    $\pi_{1, 3} \pushforward S_k \in \mathcal{T}^{opt}\left(\mu_k, \bar \mu\right)$ and
    \begin{equation}
        \int_{(\RR^n \times U)^2} d_\mathcal{U}(u_1, u_2)
        \dd S_k(x, u_1, y, u_2) < \frac{\varepsilon}{2(L_f+1)}.
    \end{equation}
    For every $k \ge \bar k$,
    define the plan $T_k \eqdef (f \otimes f) \pushforward S_k$ and note that
    $T_k \in \mathcal{T} \left(V^{\tilde u}[\mu_k], V^{\tilde u}[\bar \mu]\right)$. 
    Then,
    \begin{align*}
        & \quad W_{T\RR^n}\left(V^{\tilde u}[\mu_k], V^{\tilde u}[\bar \mu]\right)
        \\
        & \le \int_{T\RR^n \times T\RR^n} \left[\norm{x-y}_{\RR^n} + \norm{v-w}_{\RR^n}\right]
        \dd T_k(x, v, y, w)
        \\ 
        & = \int_{(\RR^n \times U)^2} \left[\norm{x-y}_{\RR^n} 
        + \norm{\hat f(x, u_1) - \hat f(y, u_2)}_{\RR^n}\right]
        \dd S_k(x, u_1, y, u_2)
        \\ 
        & \le (L_f+1) \int_{(\RR^n \times U)^2} \left[\norm{x-y}_{\RR^n} + d_{\mathcal U}(u_1, u_2)\right]
        \dd S_k(x, u_1, y, u_2)
        \\
        & \le (L_f+1) \left[W_{\RR^n}(\mu_k, \bar \mu) + \int_{(\RR^n \times U)^2} d_{\mathcal U}(u_1, u_2)
        \dd S_k(x, u_1, y, u_2)\right]
        \\
        & \le \varepsilon
    \end{align*}
    for every $k \ge \bar k$. This proves~\eqref{eq:second_estimate_thm_existence}, i.e. the second
    assumption of \cref{thm:existence-MDE} holds. The proof is completed.
\end{proof}

\section{Well-Posedness Result for Measure Controls}
\label{sec:well-posedness}
In this section we consider the 
$\mvf$ $V^{\tilde u}$, defined in~\eqref{eq:mfv-given-tilde-u}, associated
to a given measure control $\tilde{u}$. 
Under suitable assumptions on $\tilde u$, we prove the well-posedness for the
Cauchy problem associated to $V^{\tilde{u}}$.

\begin{theorem}
    \label{thm:well-posedness}
    Assume~\ref{ass:compact_controls} and~\ref{ass:control_system}.
    Consider a $\mathcal{W}$-Lipschitz continuous
    measure control $\tilde{u}$, in the sense of \cref{def:regularity_measure_control}. 
    Then, there exists a Lipschitz semigroup of solutions to 
    $\dot{\mu}=V^{\tilde{u}}[\mu]$.
\end{theorem}

\begin{proof}
    We need to prove that the vector field $V^{\tilde{u}}$ satisfies
    the assumptions of \cref{thm:wellposedness_MDE}.
    The first assumption of \cref{thm:wellposedness_MDE} coincides with the
    first assumption of \cref{thm:existence-MDE}. So it is automatically satisfied, since the
    assumptions here are stronger than those of \cref{thm:existence-MDE}.

    We need to prove the second assumption of \cref{thm:existence-MDE}.
    Since $\tilde u$ is a $\mathcal{W}$-Lipschitz continuous measure control,
    there exists a constant $C>0$ such that
    \begin{equation}
        \label{eq:ass_thm_wellposedness}
        \mathcal{W}_{\RR^n \times U} (\tilde u[\mu], \tilde u[\nu]) \le C\, W_{\RR^n}(\mu, \nu)
    \end{equation}
    for every $\mu, \nu \in \mathcal{P}_c(\RR^n)$.
    Fix $\varepsilon > 0$, $R>0$, and take $\mu, \nu \in \mathcal{P}_c(\RR^n)$ such that
    $\spt(\mu) \subseteq B(0, R)$ and $\spt(\nu) \subseteq B(0, R)$.
    Take a plan $S \in \mathcal T(\tilde u[\mu], \tilde u[\nu])$ such that
    $\pi_{1,3} \pushforward S \in \mathcal{T}^{opt} (\mu, \nu)$ and
    \begin{equation}
        \label{eq:eps_approx_W}
        \int_{(\RR^n \times U)^2} 
        d_{\mathcal{U}}(u_1, u_2) \dd S(x, u_1, y, u_2)
        < \mathcal{W}_{\RR^n \times U} \left(\tilde{u}[\mu], \tilde{u}[\nu]\right) + \varepsilon.
    \end{equation}
    Define $T \eqdef (f \otimes f) \pushforward S$. 
    Since $T \in \mathcal{T}\left(V^{\tilde u} [\mu], V^{\tilde u} [\nu]\right)$, 
    \eqref{eq:ass_thm_wellposedness}, and~\eqref{eq:eps_approx_W}, then
    \begin{align*}
        \mathcal W_{T\RR^n}\left(V^{\tilde u} [\mu], V^{\tilde u} [\nu]\right)
        & \le \int_{T\RR^n \times T \RR^n} \norm{v - w}_{\RR^n} \dd T(x, v, y, w)
        \\
        & = \int_{(\RR^n \times U)^2} \norm{\hat f(x, u_1) - \hat f(y, u_2)}_{\RR^n} \dd S(x, u_1, y, u_2)
        \\
        & \le L_f \int_{(\RR^n \times U)^2} 
        \left(\norm{x-y}_{\RR^n} + d_{\mathcal{U}}(u_1, u_2)\right) \dd S(x, u_1, y, u_2)
        \\
        & = L_f \left[W_{\RR^n}(\mu, \nu) + \int_{(\RR^n \times U)^2} 
        d_{\mathcal{U}}(u_1, u_2) \dd S(x, u_1, y, u_2)\right]
        \\
        & \le L_f \left[W_{\RR^n}(\mu, \nu) + \mathcal{W}_{\RR^n \times U} \left(\tilde{u}[\mu], \tilde{u}[\nu]\right) 
        + \varepsilon\right]
        \\
        & \le L_f(C+1) W_{\RR^n}(\mu, \nu) + L \varepsilon,
    \end{align*}
    where $L_f$ is the Lipschitz constant of the vector field $f$; see \ref{ass:control_system}.
    The arbitrariness of $\varepsilon$ permits to conclude the proof.
\end{proof}

\section{Stability of Trajectories}
\label{sec:stability}

In this section we prove stability of trajectories with respect to measure controls.
More precisely, given a convergent sequence of measure controls,
the corresponding sequence of measure trajectories converges to the trajectory associated to
the limit control.

\begin{theorem}
    \label{tmh:stability}
    Assume~\ref{ass:compact_controls} and~\ref{ass:control_system},
    and fix $\bar \mu \in \PP_c(\RR^n)$.
    Consider a sequence of uniformly
    $\mathcal{W}$-Lipschitz continuous
    measure controls $\tilde{u}_k$, in the sense of \cref{def:unif-Lip-measure-controls}.
    Assume there exists a measure control $\tilde{u}$ such that
    \begin{equation}
        \label{eq:u_k convergence}
        \lim_{k \to + \infty}
        \sup_{\mu \in \PP_c(\RR^n)}\mathcal{W}_{\R^n\times U} (\tilde{u}_k[\mu],\tilde{u}[\mu]) = 0.    
    \end{equation}
    Let $\mu_k$ be a solution to
    $\dot{\mu}=V^{\tilde{u}_k}[\mu]$ with 
    $\mu(0)=\bar \mu$ obtained as limit of LAS,
    and let $\mu$ be a solution to
    $\dot{\mu}=V^{\tilde{u}}[\mu]$ with 
    $\mu(0)=\bar \mu$ obtained as limit of LAS.
    Then, for every $t \ge 0$,
    \begin{equation}
        \label{eq:uniform-limit-trajectories}
        \lim_{k \to + \infty} W_{\RR^n} \left(\mu_k(t), \mu(t)\right) = 0.    
    \end{equation}
\end{theorem}

\begin{proof}
    Since $\tilde u_k$ is a sequence of uniformly $\mathcal{W}$-Lipschitz
    continuous measure controls, then every $\tilde u_k$ ($k \in \NN$)
    is $\mathcal{W}$-Lipschitz continuous in the sense 
    of \cref{def:regularity_measure_control}.
    Moreover, using also assumption~\eqref{eq:u_k convergence}, we can deduce that
    $\tilde u$ is $\mathcal{W}$-Lipschitz continuous too.
    Hence, by~\ref{ass:compact_controls} and~\ref{ass:control_system},
    $\tilde u$ and, for every $k \in \NN$, $\tilde u_k$ both satisfy all the
    requirements of \cref{thm:well-posedness}, so that the trajectories $\mu$ and
    $\mu_k$ are well-defined.
    
    Fix $\varepsilon > 0$ and $T > 0$. 
    Denote with $L_f$ the Lipschitz constant of the vector field $f$, see \ref{ass:control_system},    
    and call $\mathcal K \eqdef (3L_f+1)^T$.
    We approximate the solutions of the measure valued equations using the 
    \textsl{Lattice Approximate Solution} (LAS) technique; see~\cite[Definition~3.1]{piccoli_measure_2019}. 
    To this aim, fix $N \in \mathbb{N} \setminus\{0\}$.     
    By~\eqref{eq:u_k convergence}, there exists $\bar k$
    such that
    \begin{equation}
        \label{eq:u_k_epsilon}
        \mathcal{W}_{\R^n\times U} (\tilde{u}_k[\mu],\tilde{u}[\mu]) < \frac{\varepsilon}{\mathcal K^N}
    \end{equation}
    for every $k \ge \bar k$ and $\mu \in \mathcal{P}_c(\RR^n)$. 
    
    Fix $k \ge \bar k$ and define the approximated initial conditions as
    \begin{equation}
        \mu^N_k(0) \eqdef \mathcal A_N(\bar \mu)
        \qquad \text{ and } \qquad
        \mu^N(0) \eqdef \mathcal A_N(\bar \mu),
    \end{equation}
    where the operator $\mathcal{A}_N\colon \mathcal{P}_c(\RR^n) \to \mathcal{P}_c(\RR^n)$ is defined by
    \begin{equation}
        \mathcal{A}_N (\mu) = \sum_i \mu(x_i + Q) \delta_{x_i}
    \end{equation}
    for every $\mu \in \mathcal{P}_c(\RR^n)$, where $\delta_{x_i}$ is the Dirac delta centered in $x_i$,
    $Q = {\left[0, \frac{1}{N^2}\right[\,}^n$ and
    $x_i \in \frac{\mathbb{Z}^n}{N^2} \cap {[-N, N]}^n$; see~\cite[Eq.~(6)]{piccoli_measure_2019}.
    Note that $\mu^N_k(0) = \mu^N(0)$.

    We construct two time-dependent measures $\mu_k^N$ and $\mu^N$ recursively on the time interval $[0, T]$.

    Assume first that $t \in \left[0, \frac{1}{N}\right]$ and define
    \begin{align*}
        \mu_k^N \left(t\right)
        & \eqdef \sum_i \sum_j m_{ij}^v \left(V^{\tilde{u}_k}[\mu_k^N(0)]\right) \delta_{x_i + t v_j}
        \\
        \mu^N \left(t\right)
        & \eqdef \sum_i \sum_j m_{ij}^v \left(V^{\tilde{u}}[\mu^N(0)]\right) \delta_{x_i + t v_j},
    \end{align*}
    where
    \begin{align*}
        m_{ij}^v \left(V^{\tilde{u}_k}[\mu_k^N(0)]\right) 
        & =
        V^{\tilde{u}_k}[\mu_k^N(0)] \left(\left\{(x_i, v) \in T_{x_i}\RR^n \colon v \in v_j + Q'\right\}\right)
        \\
        m_{ij}^v \left(V^{\tilde{u}}[\mu^N(0)]\right) 
        & =
        V^{\tilde{u}}[\mu^N(0)] \left(\left\{(x_i, v) \in T_{x_i}\RR^n \colon v \in v_j + Q'\right\}\right)
    \end{align*}
    and $Q' = {\left[0, \frac{1}{N}\right[\,}^n$.
    For $t \in \left[0, \frac{1}{N}\right]$,
    we estimate $W_{\RR^n}\left(\mu_k^N \left(t\right), \mu^N \left(t\right)\right)$.
    By~\eqref{eq:u_k_epsilon}, there exists a plan 
    $\widetilde T^k \in \plan{\tilde{u}_k[\mu^N_k(0)]}{\tilde{u}[\mu^N(0)]}$
    such that the pushforward $\pi_{1, 3} \pushforward \widetilde{T}^k \in \optplan{\mu^N_k(0)}{\mu^N(0)}$ and
    \begin{equation}
        \label{eq:W-distance-1}
        \int_{(\RR^n \times U)^2} d_{\mathcal U}\left(u_1, u_2\right) 
        \dd{} \widetilde{T}^k (x, u_1, y, u_2) < \frac{\varepsilon}{\mathcal K^N}.
    \end{equation}
    For $t \in \left[0, \frac{1}{N}\right]$, define the map $Y_t\colon T\RR^n \times T\RR^n \to \RR^n \times \RR^n$
    as $Y_t ((x, v), (y, w)) = (x + t v, y + t w)$. Moreover, define
    $\widehat T^k \eqdef Y_t \pushforward (f \otimes f) \pushforward \widetilde{T}^k$. Clearly
    $\widehat{T}^k \in \plan{\mu^N_k(t)}{\mu^N(t)}$ and so, using~\eqref{eq:W-distance-1}
    and~\ref{ass:control_system},
    \begin{align}
        \nonumber
        & \quad W_{\RR^n}\left(\mu^N_k(t), \mu^N(t)\right)
        \\
        \nonumber
        & \le \int_{\RR^n \times \RR^n} \norm{x-y}_{\RR^n} \dd{} \widehat{T}^k(x, y)
        \\
        \nonumber
        & = \int_{T\RR^n \times T\RR^n} \norm{x + t v - y - t w}_{\RR^n} \dd{} ((f \otimes f) \pushforward \widetilde{T}^k)(x, v, y, w)
        \\
        \nonumber
        & \le \int_{T\RR^n \times T\RR^n} \left[\norm{x-y}_{\RR^n} + \norm{v-w}_{\RR^n}\right] 
        \dd{} ((f \otimes f) \pushforward \widetilde{T}^k)(x, v, y, w)
        \\
        \nonumber
        & = \int_{(\RR^n \times U)^2} \left[\norm{x-y}_{\RR^n} + \norm{\hat f(x, u_1) - \hat f(y, u_2)}_{\RR^n}\right] 
        \dd{} \widetilde{T}^k(x, u_1, y, u_2)
        \\
        \nonumber
        & \le \int_{(\RR^n \times U)^2} \left[(L_f+1)\norm{x-y}_{\RR^n} 
        + L_f d_{\mathcal{U}} (u_1, u_2)\right] 
        \dd{} \widetilde{T}^k(x, u_1, y, u_2)
        \\
        \nonumber
        & = (L_f+1) W_{\RR^n}(\mu_k^N(0), \mu^N(0)) + \frac{\varepsilon L_f}{\mathcal K^N}
        \\
        \label{eq:estimate-first-step}
        & = \frac{\varepsilon L_f}{\mathcal K^N}.
    \end{align}

    Assume now that the approximated LAS measures $\mu^N_k$ and $\mu^N$
    are defined in the time interval $\left[0, \frac{\ell}{N}\right]$ for some $\ell \in \mathbb{N}\setminus\{0\}$.
    We extend now these measures to the time interval $\left[0, \frac{\ell+1}{N}\right]$.
    For $t \in \left[\frac{\ell}{N}, \frac{\ell+1}{N}\right]$ we define
    \begin{align*}
        \mu_k^N \left(t\right)
        & \eqdef \sum_i \sum_j m_{ij}^v \left(V^{\tilde{u}_k}\left[\mu_k^N\left(\frac{\ell}{N}\right)\right]\right) 
        \delta_{x_i + t v_j}
        \\
        \mu^N \left(t\right)
        & \eqdef \sum_i \sum_j m_{ij}^v \left(V^{\tilde{u}}\left[\mu^N\left(\frac{\ell}{N}\right)\right]\right) 
        \delta_{x_i + t v_j},
    \end{align*}
    where
    \begin{align*}
        m_{ij}^v \left(V^{\tilde{u}_k}\left[\mu_k^N\left(\frac{\ell}{N}\right)\right]\right) 
        & =
        V^{\tilde{u}_k}\left[\mu_k^N\left(\frac{\ell}{N}\right)\right]
        \left(\left\{(x_i, v) \in T_{x_i}\RR^n \colon v \in v_j + Q'\right\}\right)
        \\
        m_{ij}^v \left(V^{\tilde{u}}\left[\mu^N\left(\frac{\ell}{N}\right)\right]\right) 
        & =
        V^{\tilde{u}}\left[\mu^N\left(\frac{\ell}{N}\right)\right]
        \left(\left\{(x_i, v) \in T_{x_i}\RR^n \colon v \in v_j + Q'\right\}\right)
    \end{align*}
    and $Q' = {\left[0, \frac{1}{N}\right[\,}^n$.
    
    Define $\widetilde{S}_k \in \plan{\tilde{u}_k\left[\mu_k^N\left(\frac{\ell}{N}\right)\right]}
    {\tilde{u}_k\left[\mu^N\left(\frac{\ell}{N}\right)\right]}$ as a plan such that
    $\pi_{1, 3} \pushforward \widetilde{S}_k \in \optplan{\mu_k^N\left(\frac{\ell}{N}\right)}{\mu^N\left(\frac{\ell}{N}\right)}$
    and
    \begin{equation}
        \label{eq:estimate-Sk}
        \int_{(\RR^n \times U)^2} d_{\mathcal U} \left(u_1, u_2\right)
        \dd{} \widetilde{S}_k(x, u_1, y, u_2)
        \le 
        W_{\RR^n}\left(\mu_k^N\left(\frac{\ell}{N}\right), \mu^N\left(\frac{\ell}{N}\right)\right) 
        + \frac{\varepsilon}{\mathcal K^N}.    
    \end{equation}

    By~\eqref{eq:u_k_epsilon}, 
    there exists a plan $\widetilde{T}_k \in \plan{\tilde{u}_k\left[\mu^N\left(\frac{\ell}{N}\right)\right]}
    {\tilde{u}\left[\mu^N\left(\frac{\ell}{N}\right)\right]}$ such that
    $\pi_{1, 3} \pushforward \widetilde{T}_k \in \optplan{\mu^N\left(\frac{\ell}{N}\right)}{\mu^N\left(\frac{\ell}{N}\right)}$
    and
    \begin{equation}
        \label{eq:estimate-Tk}
        \int_{(\RR^n \times U)^2} d_{\mathcal U} \left(u_1, u_2\right)
        \dd{} \widetilde{T}_k(x, u_1, y, u_2) < \frac{\varepsilon}{\mathcal K^N}.
    \end{equation}

    By Gluing Lemma (see~\cite[Lemma~5.3.2]{ambrosio-gigli-savare-book} or~\cite[Chapter~1]{villani2008optimal}),
    there exists $\widetilde{R}_k \in \mathcal{P}\left((\RR^n \times U)^3\right)$ such that
    $\pi_{1,2} \pushforward \widetilde{R}_k = \widetilde{S}_k$, $\pi_{2,3} \pushforward \widetilde{R}_k = \widetilde{T}_k$,
    and $\pi_{1,3} \pushforward \widetilde{R}_k \in \plan{\tilde{u}_k\left[\mu^N_k\left(\frac{\ell}{N}\right)\right]}
    {\tilde{u}\left[\mu^N\left(\frac{\ell}{N}\right)\right]}$.
    Finally, define
    $\widehat Z^k \eqdef Y_t \pushforward (f \otimes f) \pushforward \pi_{1, 3} \pushforward \widetilde{R}^k$. Clearly
    $\widehat{Z}^k \in \plan{\mu^N_k\left(\frac{\ell}{N}\right)}{\mu^N\left(\frac{\ell}{N}\right)}$ and so,
    using~\eqref{eq:estimate-Sk}, \eqref{eq:estimate-Tk}, and~\ref{ass:control_system},
    \begin{align}
        \nonumber
        & \quad
        W_{\RR^n}\left(\mu^N_k(t), \mu^N(t)\right)
        \\
        \nonumber
        & \le \int_{\RR^n \times \RR^n} \norm{x-y}_{\RR^n} \dd{} \widehat{Z}^k(x, y)
        \\
        \nonumber
        & = \int_{T\RR^n \times T\RR^n} \norm{x + t v - y - t w}_{\RR^n} 
        \dd{} ((f \otimes f) \pushforward \pi_{1, 3} \pushforward 
        \widetilde{R}^k)(x, v, y, w)
        \\
        \nonumber
        & \le \int_{T\RR^n \times T\RR^n} \left[\norm{x-y}_{\RR^n} + \norm{v-w}_{\RR^n}\right]
        \dd{} ((f \otimes f) \pushforward \pi_{1, 3} \pushforward
        \widetilde{R}^k)(x, v, y, w)
        \\
        \nonumber
        & = \int_{(\RR^n \times U)^2} \left[\norm{x-y}_{\RR^n} 
        + \norm{\hat f(x, u_1) - \hat f(y, u_2)}_{\RR^n}\right]
        \dd{} (\pi_{1, 3} \pushforward
        \widetilde{R}^k)(x, u_1, y, u_2)
        \\
        \nonumber
        & \le \int_{(\RR^n \times U)^2} \left[(L_f+1) \norm{x-y}_{\RR^n}
        + L_f d_{\mathcal{U}} (u_1, u_2)\right] 
        \dd{} (\pi_{1, 3} \pushforward \widetilde{R}^k) (x, u_1, y, u_2)
        \\
        \nonumber
        & \le \int_{(\RR^n \times U)^3} \left[(L_f+1) \norm{x-z}_{\RR^n}
        + L_f d_{\mathcal{U}} (u_1, u_3)\right] 
        \dd{} \widetilde{R}^k (x, u_1, z, u_3, y, u_2)
        \\
        \nonumber
        & \quad + \int_{(\RR^n \times U)^3} \left[(L_f+1) \norm{z-y}_{\RR^n}
        + L_f d_{\mathcal{U}} (u_3, u_2)\right] 
        \dd{} \widetilde{R}^k (x, u_1, z, u_3, y, u_2)
        \\
        \nonumber
        & = \int_{(\RR^n \times U)^2} \left[(L_f+1) \norm{x-z}_{\RR^n}
        + L_f d_{\mathcal{U}} (u_1, u_3)\right] 
        \dd{} \widetilde{S}^k (x, u_1, z, u_3)
        \\
        \nonumber
        & \quad + \int_{(\RR^n \times U)^2} \left[(L_f+1) \norm{z-y}_{\RR^n}
        + L_f d_{\mathcal{U}} (u_3, u_2)\right] 
        \dd{} \widetilde{T}^k (z, u_3, y, u_2)
        \\
        \nonumber
        & < 2 (L_f+1) W_{\RR^n} \left(\mu_k^N\left(\frac{\ell}{N}\right), \mu^N\left(\frac{\ell}{N}\right)\right)
        \\
        \nonumber
        & \quad
        + L_f W_{\RR^n} \left(\mu_k^N\left(\frac{\ell}{N}\right), \mu^N\left(\frac{\ell}{N}\right)\right) 
        + \frac{2 L \varepsilon}{\mathcal K^N}
        \\
        \label{eq:W-recursive}
        & < (3L_f + 1) W_{\RR^n}\left(\mu_k^N\left(\frac{\ell}{N}\right), \mu^N\left(\frac{\ell}{N}\right)\right) 
        + \frac{2L\varepsilon}{\mathcal K^N}
    \end{align}
    for every $t \in \left[\frac{\ell}{N}, \frac{\ell+1}{N}\right]$.

    Using~\eqref{eq:estimate-first-step} and the recursive estimate~\eqref{eq:W-recursive},
    for every $t \in \left[\frac{\ell}{N}, \frac{\ell+1}{N}\right]$, we have that
    \begin{align*}
        W_{\RR^n}\left(\mu^N_k(t), \mu^N(t)\right) 
        & < (3L_f+1)^\ell W_{\RR^n}\left(\mu^N_k\left(\frac{1}{N}\right), 
        \mu^N\left(\frac{1}{N}\right)\right)
        \\
        & \quad
        + \frac{2L_f \varepsilon}{\mathcal K^N} \sum_{i=0}^{\ell - 1}(3L+1)^i
        \\
        & < (3L_f+1)^\ell \frac{L_f \varepsilon}{\mathcal K^N} 
        + \frac{2L_f \varepsilon}{\mathcal K^N} \sum_{i=0}^{\ell - 1}(3L_f+1)^i
        \\
        & \le \frac{2L_f \varepsilon}{\mathcal K^N} \sum_{i=0}^{\ell}(3L_f+1)^i
        \\
        & = \frac{2 \varepsilon}{\mathcal K^N} \, \frac{(3L_f+1)^{\ell + 1} - 1}{3}.
    \end{align*}
    Since $t \le T$, then the index $\ell$ is lower than or equal to $TN - 1$. Thus,
    for every $t \in [0, T]$,
    \begin{equation}
        W_{\RR^n}\left(\mu^N_k(t), \mu^N(t)\right) < \frac{2 \varepsilon \mathcal K^N}{3 \mathcal K^N} 
        = \frac{2 \varepsilon}{3} .
    \end{equation}
    The measures $\mu^N_k$ and $\mu^N$ converge uniformly on $[0, T]$ as $N \to +\infty$ with respect
    to the Wasserstein metric;
    see~\cite[Proof of Theorem~3.1]{piccoli_measure_2019}.
    Denote respectively $\mu_k$ and $\mu$ the limits. Hence,
    for every $t \in [0, T]$ and $k \ge \bar k$,
    \begin{equation}
        W_{\RR^n}\left(\mu_k(t), \mu(t)\right) < \varepsilon,
    \end{equation}    
    which implies~\eqref{eq:uniform-limit-trajectories} for every $t \in [0, T]$
    and, since $T$ is arbitrary, for every $t \ge 0$.
\end{proof}

\section{Closure of Trajectories}
\label{sec:closure}
In this section, we prove that the set of trajectories is closed.
More precisely, given a sequence of uniformly $\mathcal W$-Lipschitz
continuous measure controls (in the sense of \cref{def:unif-Lip-measure-controls}),
then the corresponding trajectories converge to a solution of a suitable
measure valued equation. 

\begin{theorem}
    \label{thm:closure}
    Assume~\ref{ass:compact_controls} and~\ref{ass:control_system},
    and fix $\bar \mu \in \PP_c(\RR^n)$ and a compact set $K \subseteq \R^n$ such that
    $\spt(\bar \mu) \subseteq K$.
    Consider a sequence of uniformly
    $\mathcal{W}$-Lipschitz continuous
    measure controls $\tilde{u}_k$, in the sense of \cref{def:unif-Lip-measure-controls}.

    Let $\mu_k$ be a solution to
    $\dot{\mu}=V^{\tilde{u}_k}[\mu]$ with 
    $\mu(0) = \bar \mu$ obtained as limit of LAS.
    Assume moreover that
    $\spt(\mu_k) \subseteq K$
    for every $k \in \N$.

    Suppose there exists a map
    $t \mapsto \mu(t) \in \PP(\R^n)$ such that
    $W_{\RR^n}(\mu_k(t), \mu(t)) \to 0$ as
    $k \to + \infty$ uniformly in $t$, i.e.
    \begin{equation}
        \label{eq:W-uniform-control}
        \lim_{k \to + \infty} \, \sup_{t \ge 0}\, 
        W_{\RR^n}(\mu_k(t), \mu(t)) = 0.
    \end{equation}

    Then, the map $t \mapsto \mu(t)$
    is a solution to $\dot \mu = V[\mu]$, $\mu(0) = \bar \mu$,
    for some measure vector field $V$.
\end{theorem}

\begin{proof}
    Since $\tilde u_k$ is a uniformly $\mathcal W$-Lipschitz continuous sequence
    of measure controls, then there exists a constant $C>0$ such that,
    for every $\mu, \nu \in \PP(\R^n)$ and $k \in \N$,
    \begin{equation}
        \label{eq:Lip-control-uniform}
        \mathcal{W}_{\R^n \times U} (\tilde{u}_k[\mu], \tilde{u}_k[\nu]) 
        \le C W_{\RR^n}(\mu, \nu).
    \end{equation}

    \noindent
    Fix $g \in C_c^\infty(\RR^n)$.
    We split the proof in several steps.\\

    \noindent
    \textbf{Step 1: $\mu(0) = \bar \mu$}.\\ 
    For every $k \in \NN$, since $\mu_k$ satisfies $\mu_k(0) = \bar \mu$ by assumptions,
    then 
    \begin{align*}
        W_{\RR^n}\left(\mu(0), \bar \mu\right)
        & \le W_{\RR^n}\left(\mu(0), \mu_k(0)\right) + W_{\RR^n}\left(\mu_k(0), \bar \mu\right)
        = W_{\RR^n}\left(\mu(0), \mu_k(0)\right)
        \\
        & \le \sup_{t \ge 0} W_{\RR^n}\left(\mu(t), \mu_k(t)\right).
    \end{align*}
    Passing to the limit as $k \to +\infty$ in the previous inequality,  
    since~\eqref{eq:W-uniform-control}, we deduce that
    $W_{\RR^n}\left(\mu(0), \bar \mu\right) = 0$, proving the claim.\\

    \noindent
    \textbf{Step 2: the map $t \mapsto \int_{\RR^n} g(x) d (\mu(t))(x)$ is absolutely continuous}.\\
    For every $k \in \mathbb{N}$,
    since $\mu_k$ is a solution to $\dot \mu = V^{\tilde u_k}[\mu]$, 
    using~\eqref{eq:mfv-given-tilde-u} and $\spt(\mu_k) \subseteq K$,
    we obtain that
    \begin{align}
        \nonumber
        \abs{\frac{d}{d t} \int_{\RR^n} g(x) \dd{} (\mu_k(t))(x)}
        & = \abs{\int_{T\RR^n} \left(\nabla g(x) \cdot v\right)\,  \dd{} \left(V^{\tilde u_k}[\mu_k(t)]\right) (x, v) }
        \\
        \nonumber
        & = \abs{\int_{T\RR^n} \left(\nabla g(x) \cdot v\right)\,  
        \dd{} \left(f \pushforward \tilde u_k [\mu_k(t)]\right) (x, v) }
        \\
        \nonumber
        & = \abs{\int_{\RR^n \times U} \left(\nabla g(x) \cdot \hat f(x, u)\right)\,  
        \dd{} \left(\tilde u_k [\mu_k(t)]\right) (x, v) }
        \\
        \nonumber
        & = \abs{\int_{K \times U} \left(\nabla g(x) \cdot \hat f(x, u)\right)\,  
        \dd{} \left(\tilde u_k [\mu_k(t)]\right) (x, v) }
        \\
        \label{eq:2}
        & \le \norm{\nabla g}_{L^\infty(K)}
        \norm{\hat f}_{L^\infty(K \times U)}.
    \end{align}
    By \ref{ass:control_system}, we deduce that, for every $k \in \NN$, the map
    \begin{equation*}
        t \mapsto \int_{\RR^n} g(x) \dd{} (\mu_k(t))(x)
    \end{equation*}
    is uniformly Lipschitz continuous with Lipschitz constant bounded by the constant
    $\norm{\nabla g}_{L^\infty(K)} \norm{\hat f}_{L^\infty(K \times U)}$, which does not
    depend on $k$.
    Hence, by~\eqref{eq:W-uniform-control}, also the map
    \begin{equation*}
        t \mapsto \int_{\RR^n} g(x) \dd{} (\mu(t))(x)
    \end{equation*}
    is Lipschitz continuous, and, consequently, absolutely continuous.\\

    \noindent
    \textbf{Step 3: compactness for $V^{\tilde u_k}$}.\\
    Define the metric spaces $\left(X, W_{\RR^n}\right)$ and $(Y, W_{T\RR^n})$, where
    \begin{equation*}
        \begin{split}
            X & \eqdef \left\{\eta \in \PP(\RR^n) \colon \spt(\eta) \subseteq K\right\},
            \\
            Y & \eqdef \left\{\eta \in \PP(T\RR^n) \colon \spt(\eta) \subseteq f\left(K, U\right)\right\}.
        \end{split}        
    \end{equation*}
    Consider now the sequence of measure vector fields
    $V^{\tilde u_k} \colon \PP(\RR^n) \to \PP(T\RR^n)$
    and note that, restricting the domain to $X$, we can assume that
    $V^{\tilde u_k} \colon X \to Y$.

    Fix $\mu^1, \mu^2 \in X$ and $k \in \NN$. 
    Take $S \in \optplan{\tilde u_k[\mu^1]}{\tilde u_k[\mu^2]}$ and consider
    $R = (f \times f) \pushforward S$, which is a transference plan for 
    $V^{\tilde{u}_k}[\mu^1]$ and 
    $V^{\tilde{u}_k}[\mu^2]$. Hence, using~\cite[Lemma~4.1]{piccoli_measure_2019}
    and~\eqref{eq:Lip-control-uniform},
    \begin{align}
        \nonumber
        \quad W_{T\RR^n} \left(V^{\tilde{u}_k}[\mu^1], V^{\tilde{u}_k}[\mu^2]\right)
        & \le \int_{T\RR^n \times T\RR^n} d_{T\RR^n} \left((x, v), (y, w)\right) \dd{}R(x, v, y, w)
        \\
        \nonumber
        & = \int_{(\RR^n \times U)^2} d_{\RR^n \times U} \left((x, u_1), (y, u_2)\right) \dd{}S(x, u_1, y, u_2)
        \\
        \nonumber
        & = W_{\RR^n \times U} \left(\tilde u_k[\mu^1], \tilde u_k[\mu^2]\right)
        \\
        \nonumber
        & \le \mathcal W_{\RR^n \times U} \left(\tilde u_k[\mu^1], \tilde u_k[\mu^2]\right)
        + W_{\RR^n}\left(\mu^1, \mu^2\right)
        \\
        \label{eq:3}
        & \le (C+1) W_{\RR^n}\left(\mu^1, \mu^2\right).
    \end{align}
    This proves that, for every $k \in \NN$, $V^{\tilde u_k} \in C^0\left(X; Y\right)$ 
    and that the sequence $V^{\tilde u_k}$ is equicontinuous.

    Given $\mu^1 \in X$, consider the set
    \begin{equation*}
        \mathcal F_{\mu^1} \eqdef \left\{V^{\tilde u_k} [\mu^1] \colon k \in \NN\right\}.
    \end{equation*}
    Clearly, every element of $\mathcal{F}_{\mu^1}$ is a probability measure with support
    contained in the compact set $f(K, U)$. This implies that $\mathcal{F}_{\mu^1}$ is tight;
    see~\cite[Formula~(5.1.8)]{ambrosio-gigli-savare-book}. Hence, Prokhorov Theorem 
    (see~\cite[Theorem~5.1.3]{ambrosio-gigli-savare-book}) implies that $\mathcal{F}_{\mu^1}$
    is relatively compact in $Y$. 
    Thus, applying Ascoli-Arzelà Theorem (see~\cite[Theorem~6.1, Chapter~7]{munkres1975topology}
    or~\cite[Theorem~47.1]{munkres2013topology}),
    there exists $V \in C^0\left(X; Y\right)$ and a subsequence 
    $\tilde u_{k_h}$ such that
    \begin{equation}
        \label{eq:5}
        \lim_{h \to +\infty} \sup_{\mu \in X}
        W_{T\RR^n} \left(V^{\tilde u_{k_h}}[\mu], V[\mu]\right) = 0.
    \end{equation}

    \noindent
    \textbf{Step 4: convergence of $\int_{T\R^n} \left(\nabla g(x) \cdot v\right)\, 
    \dd{} V^{\tilde u_{k_h}}[\mu_{k_h}(t)](x, v)$}.\\
    Using~\cite[Formula~(7.1.2)]{ambrosio-gigli-savare-book}
    and~\eqref{eq:3}, for $t \ge 0$, we have
    \begin{align}
        \nonumber
        & \quad \norm{\int_{T\R^n} \nabla g(x) \cdot v\, \dd{} V^{\tilde u_{k_h}}[\mu_{k_h}(t)](x, v)
        -\int_{T\R^n} \nabla g(x) \cdot v\, \dd{} V[\mu(t)](x, v)}
        \\
        \nonumber
        & \le \norm{\int_{T\R^n} \nabla g(x) \cdot v\, \dd{} V^{\tilde u_{k_h}}[\mu_{k_h}(t)](x, v)
        -\int_{T\R^n} \nabla g(x) \cdot v\, \dd{} V^{\tilde u_{k_h}}[\mu(t)](x, v)}
        \\
        \nonumber
        & \quad + \norm{\int_{T\R^n} \nabla g(x) \cdot v\, \dd{} V^{\tilde u_{k_h}}[\mu(t)](x, v)
        -\int_{T\R^n} \nabla g(x) \cdot v\, \dd{} V [\mu(t)](x, v)}
        \\
        \nonumber
        & \le \lip(\nabla g(x) \cdot v) W_{T\RR^n} \left(V^{\tilde u_{k_h}}[\mu_{k_h}(t)], 
        V^{\tilde u_{k_h}}[\mu(t)]\right)
        \\
        \nonumber
        & \quad + \lip(\nabla g(x) \cdot v) W_{T\RR^n} \left(V^{\tilde u_{k_h}}[\mu (t)], 
        V [\mu(t)]\right)
        \\
        \nonumber
        & \le \lip(\nabla g(x) \cdot v) (C+1) W_{\RR^n} \left(\mu_{k_h}(t), \mu(t)\right)
        \\
        \label{eq:4}
        & \quad + \lip(\nabla g(x) \cdot v) W_{T\RR^n} \left(V^{\tilde u_{k_h}}[\mu (t)], 
        V [\mu(t)]\right),
    \end{align}
    where $\lip(\nabla g(x) \cdot v)$ denotes the Lipschitz constant of the map
    $(x, v) \mapsto \nabla g(x) \cdot v$.

    Using assumption~\eqref{eq:W-uniform-control} and the property~\eqref{eq:5},
    and passing to the limit as $h \to +\infty$ in~\eqref{eq:4}, we deduce that
    \begin{align*}
        \lim_{h \to +\infty}
        \int_{T\R^n} \left(\nabla g(x) \cdot v\right)\, \dd{} V^{\tilde u_{k_h}}[\mu_{k_h}(t)](x, v)
        = \int_{T\R^n} \left(\nabla g(x) \cdot v\right)\, \dd{} V[\mu(t)](x, v).
    \end{align*}
    This also proves that
    \begin{equation*}
        \int_{T\R^n} \left(\nabla g(x) \cdot v\right)\, \dd{} V[\mu(t)](x, v)
    \end{equation*}
    is defined for every $t \ge 0$.\\

    \noindent
    \textbf{Step 5: the map $t \mapsto \int_{T\R^n} \left(\nabla g(x) \cdot v\right)\, \dd{} V[\mu(t)](x, v)$
    is locally integrable}.\\
    The measurability of the map follows by the previous step, since it is the limit of measurable
    functions.

    By \textbf{Step 2}, we have that
    \begin{equation*}
        \abs{\int_{T\RR^n} \left(\nabla g(x) \cdot v\right)\,  \dd{} \left(V^{\tilde u_{k_h}}[\mu_{k_h}(t )\right)] (x, v) }
        \le \norm{\nabla g}_{L^\infty(K)}
        \norm{\hat f}_{L^\infty(K \times U)}
    \end{equation*}
    for every $h \in \NN$ and $t \ge 0$. Passing to the limit as $h \to + \infty$ we deduce that
    \begin{equation*}
        \abs{\int_{T\RR^n} \left(\nabla g(x) \cdot v\right)\,  \dd{} \left(V [\mu(t)]\right) (x, v) }
        \le \norm{\nabla g}_{L^\infty(K)}
        \norm{\hat f}_{L^\infty(K \times U)}
    \end{equation*}
    for every $t \ge 0$, proving the claim.\\

    \noindent
    \textbf{Step 6: the map $t \mapsto \mu(t)$ is a solution}.\\
    For every $h \in \mathbb{N}$,
    since $\mu_{k_h}$ is a solution to $\dot \mu = V^{\tilde u_{k_h}}[\mu]$,
    we have that
    \begin{align*}
        \int_{\RR^n} g(x) \dd{} (\mu_{k_h}(t))(x)
        & = \int_{\RR^n} g(x) \dd{} \bar \mu (x)
        \\
        & \quad
        + \int_0^t \int_{T\RR^n} \left(\nabla g(x) \cdot v\right)\,  
        \dd{} \left(V^{\tilde u_{k_h}} [\mu_{k_h}(s)]\right) (x, v) \dd s
    \end{align*}
    for every $t \ge 0$.
    Passing to the limit as $h \to +\infty$ in the previous equality and using
    the Lebesgue Theorem, we deduce that
    \begin{align*}
        \int_{\RR^n} g(x) \dd{} (\mu(t))(x)
        & = \int_{\RR^n} g(x) \dd{} \bar \mu (x)
        \\
        & \quad
        + \int_0^t \int_{T\RR^n} \left(\nabla g(x) \cdot v\right)\,  
        \dd{} \left(V [\mu(s)]\right) (x, v) \dd s
    \end{align*}
    for every $t \ge 0$. Therefore, all the requirements 
    of~\cite[Definition~2.2]{piccoli_measure_2019} are satisfied, proving that the map
    $t \mapsto \mu(t)$ is a solution to $\dot \mu = V[\mu]$, $\mu(0) = \bar \mu$, concluding
    the proof.
\end{proof}

\appendix
\section*{Appendix}

This appendix is devoted to a brief introduction to the Wasserstein distance
on Polish spaces and to the statement of the disintegration theorem for probability measures.

\section{Probability Measures and Wasserstein Distance}
\label{sec:Wasserstein}
Here we briefly recall the concept of probability measures on Polish spaces and the notion
of the Wasserstein distance; see~\cite[Chapter~5 and Chapter~7]{ambrosio-gigli-savare-book}
for a detailed description.

Let \((X, d)\) be a Polish space, that is, a metric space that is both complete and separable.  
With the symbol $\PP(X)$ we denote the space of probability measures on \(X\);  
that is, \(\mu \in \mathcal{P}(X)\) if and only if $\mu$ is a finite positive Borel
measure on $X$ such that $\mu(X) = 1$.
The support of a probability measure $\mu$ is defined by
\begin{equation}
    \label{eq:support-measure}
    \spt(\mu) \eqdef \left\{x \in X \colon \mu(U) > 0 
    \textrm{ for every neighborhood } U \textrm{ of } x \right\}.
\end{equation}
The notation $\PP_c(X)$ denotes the subset of probability measures on $X$ with compact
support.

Now let \((X_1, d_1)\) and \((X_2, d_2)\) be Polish spaces. 
For a Borel map \(\phi \colon X_1 \to X_2\) 
and a measure \(\mu \in \mathcal{P}(X_1)\), 
the pushforward of \(\mu\) under \(\phi\), denoted \(\phi \pushforward \mu\), is defined by
\begin{equation}
    \label{eq:push-forward}
    (\phi \pushforward \mu)(B) \eqdef 
    \mu(\phi^{-1}(B)) \quad \text{for every Borel set } B \subset X_2.
\end{equation}
Note that $\phi \pushforward \mu \in \PP(X_2)$.
We recall the result about disintegration of measures; 
see~\cite[Theorem~5.3.1]{ambrosio-gigli-savare-book}.
\begin{theorem} (Disintegration)
    \label{thm:disintegration}
    Let $(X, d_X)$ and $(Y, d_Y)$ be two Polish spaces,
    and fix $\mu \in \mathcal{P}(X)$ and a Borel map $r:X \to Y$.

    Then, there exists a family of Borel probability measures $\{\mu_y\}_{y \in Y}$,
    $\mu_y \in \mathcal{P}(X)$,
    uniquely determined for $(r\pushforward\mu)$-a.e. $y \in Y$, such that
    $\mu_y(X \setminus r^{-1}(y)) = 0$ for $(r\pushforward\mu)$-a.e. $y \in Y$ and
    \begin{equation*}
        \int_X \phi(x) \dd{} \mu(x) = \int_Y \int_{r^{-1}(y)} \phi(x) \dd{} \mu_y(x) \dd{}(r \pushforward \mu) (y) 
    \end{equation*}
    for every bounded Borel map $\phi: X \to \RR$.
\end{theorem}

We finally introduce the Wasserstein distance.
Let \(\mu_1, \mu_2 \in \mathcal{P}(X)\) be two probability measures 
on a Polish space \((X, d)\).  
Then, the 
Wasserstein distance between \(\mu_1\) and \(\mu_2\) is defined by 
\begin{equation}
    \label{eq:Wasserstein-distance}
    W_X (\mu_1, \mu_2) \eqdef \min \left\{\int_{X\times X} d(x, y) \, \dd{}\mu(x,y) 
    \colon \mu \in \plan{\mu_1}{\mu_2}\right\},
\end{equation}
where
\begin{equation}
    \label{eq:plans}
    \plan{\mu_1}{\mu_2} \eqdef \left\{\mu \in \PP(X \times X) \colon 
    \pi_1 \pushforward \mu = \mu_1, \, \pi_2 \pushforward \mu = \mu_2\right\}
\end{equation}
is the set of transference plans between \(\mu_1\) and \(\mu_2\), where $\pi_i$ denotes
the $i$-th projection ($i=1, 2$). 
In the case of  
\(\mu_1, \mu_2 \in \PP_c(X)\), 
then we have the dual formulation  
 \begin{equation}
    \label{eq:K-R_duality}
    W_X(\mu_1, \mu_2) = \sup \left\{\int_X f(x) \dd{} (\mu_1 - \mu_2)(x)\colon f \colon X \to \R, \lip(f) \le 1\right\},
\end{equation}
where \(\lip(f)\) is the Lipschitz constant of \(f\) defined as 
\begin{align*}
   \lip(f) \eqdef \sup \limits_{x, y \in X, \, x \not = y} \frac{|f(x)-f(y)|}{d(x,y)} \,.
\end{align*}
We define the set of optimal plans as
\begin{equation}
    \label{eq:opt_plans}
    \optplan{\mu_1}{\mu_2} \eqdef \left\{\mu \in \plan{\mu_1}{\mu_2}\colon W_X(\mu_1, \mu_2) = 
    \int_{X \times X} d(x, y) \dd{}\mu(x, y)\right\}.
\end{equation}

In the case $X = T\RR^n$, we introduce the following Wasserstein pseudo distance; 
see~\cite[Definition~4.1]{piccoli_measure_2019}.
Consider $\nu_1, \nu_2 \in \mathcal P(T\RR^n)$ two probability measures on $T\RR^n$ and
define $\mu_1 = \pi_1 \pushforward \nu_1 \in \mathcal P(\RR^n)$ 
and $\mu_2 = \pi_1 \pushforward \nu_2 \in \mathcal P(\RR^n)$, where $\pi_1$ denotes the natural
projection on $\RR^n$.
We define the Wasserstein pseudo distance
\begin{equation}
    \label{eq:W_tangent}
    \mathcal W_{T\RR^n}\left(\nu_1, \nu_2\right) \eqdef 
    \inf_{T \in \mathcal{A}(\nu_1, \nu_2)} 
    \left\{\int_{T \RR^n \times T \RR^n} \norm{v-w}_{\RR^n} \dd T(x,v,y,w)
    \right\},
\end{equation}
where $\mathcal{A}(\nu_1, \nu_2) \eqdef \left\{T \in \mathcal T (\nu_1, \nu_2) \colon
        \pi_{1,3} \pushforward T \in \mathcal T^{opt}(\mu_1, \mu_2)\right\},$
$\norm{\cdot}_{\RR^n}$ is the Euclidean norm on $\RR^n$, and the map
$\pi_{1,3}: T\RR^n \times T\RR^n \to \RR^n \times \RR^n$ 
is defined by $\pi_{1,3}(x,v,y,w) = (x, y)$.

Finally, in the case $X = \RR^n \times U$, we introduce the Wasserstein pseudo distance
$\mathcal W_{\RR^n \times U}$.
Consider $\nu_1, \nu_2 \in \mathcal P(\RR^n \times U)$ and
define $\mu_1 = \pi_1 \pushforward \nu_1 \in \mathcal P(\RR^n)$ 
and $\mu_2 = \pi_1 \pushforward \nu_2 \in \mathcal P(\RR^n)$, where $\pi_1$ denotes the 
projection onto $\RR^n$.
We define the Wasserstein pseudo distance
\begin{equation}
    \label{eq:W_RRn_U}
    \mathcal W_{\RR^n \times U}\left(\nu_1, \nu_2\right) \eqdef 
    \inf_{T \in \mathcal{A}(\nu_1, \nu_2)} 
    \left\{\int_{(\RR^n \times U)^2} d_{\mathcal U}(u_1, u_2)\, \dd T(x,u_1,y,u_2)
    \right\},
\end{equation}
where $\mathcal{A}(\nu_1, \nu_2) \eqdef \left\{T \in \mathcal T (\nu_1, \nu_2) \colon
\pi_{1,3} \pushforward T \in \mathcal T^{opt}(\mu_1, \mu_2)\right\}$, 
and the map
$\pi_{1,3}: (\RR^n \times U)^2 \to \RR^n \times \RR^n$ 
is defined by $\pi_{1,3}(x,u_1,y,u_2) = (x, y)$.

We refer the readers to \cite{Santambrogio_book, villani2008optimal} for more properties of the Wasserstein distance
and to~\cite[Remark~1]{piccoli_measure_2019} for the fact that $\mathcal W_{T\RR^n}$ is not a distance.

\bibliography{ref_arxiv}
\bibliographystyle{plain}

\end{document}